\documentclass[a4paper,12pt,reqno]{amsart}
\usepackage{enumerate,amsmath,amsthm,amssymb,cite,mathrsfs,eufrak,graphicx,caption,subfigure}

\DeclareMathOperator{\re}{\mathrm{Re}}
\DeclareMathOperator{\imag}{\mathrm{Im}}
\DeclareMathOperator{\se}{\mathcal{S}^*_e}
\DeclareMathOperator{\ke}{\mathcal{K}_e}
\numberwithin{equation}{section}
\newtheorem{theorem}{Theorem}[section]
\newtheorem{lemma}[theorem]{Lemma}
\newtheorem{corollary}[theorem]{Corollary}
\theoremstyle{remark}

\newtheorem{example}[theorem]{Example}
\newtheorem{definition}[theorem]{Definition}

\makeatletter
\@namedef{subjclassname@2010}{%
	\textup{2010} Mathematics Subject Classification}
\makeatother

\begin{document}

\title[Geometric Properties of Generalized Bessel Function]{Geometric Properties of Generalized Bessel Function associated with the Exponential Function}
\author[Adiba Naz]{Adiba Naz}

\address{Department of Mathematics, University of Delhi, 	Delhi--110 007, India}
\email{adibanaz81@gmail.com}

\author[Sumit Nagpal]{Sumit Nagpal}
\address{Department of Mathematics, Ramanujan College, University of Delhi,	Delhi--110 019, India}
\email{sumitnagpal.du@gmail.com}

\author[V. Ravichandran]{V. Ravichandran}

\address{Department of Mathematics, National Institute of Technology,  Tiruchirappalli--620 015,  India}
\email{vravi68@gmail.com, ravic@nitt.edu}
\dedicatory{Dedicated to Prof.\@ Ajay Kumar on the occasion of his retirement}

\begin{abstract}
Sufficient conditions are determined on the parameters such that the  generalized  and normalized Bessel function of the first kind and other related functions belong to subclasses of starlike and convex functions defined in the unit disk associated with the exponential mapping. Several differential subordination implications are derived for analytic functions involving Bessel function and the operator introduced by Baricz \emph{et al.} [Differential subordinations involving generalized Bessel functions, Bull. Malays. Math. Sci. Soc. {\bf 38} (2015), no.~3, 1255--1280]. These results are obtained by constructing suitable class of admissible functions. Examples involving trigonometric and hyperbolic functions are provided to illustrate the obtained results.
\end{abstract}


\maketitle
\section{Introduction}
Due to diverse applications of Bessel function in wave propagation and static potential, it is widely studied in geometric function theory. Consider the generalized Bessel differential equation
\begin{equation}\label{e1}
z^2\omega''(z)+bz\omega'(z)+(cz^2-\nu^2+\nu(1-b))\omega(z)=0 \qquad(z,\nu,b,c\in\mathbb{C}).
\end{equation}The particular solution of \eqref{e1} is known as the generalized Bessel function $\omega_{\nu,b,c}$ of the first kind  of order $\nu$ having the infinite series representation
\begin{equation}\label{e6}
\omega_{\nu,b,c}(z)=\sum_{n\geq 0} \frac{(-c)^n}{n! \Gamma (\nu+n+(b+1)/2)}\left(\frac{z}{2}\right)^{2n+\nu}\qquad(z,\nu,b,c\in\mathbb{C})
\end{equation}
where $\Gamma$ denotes the Euler gamma function. Set $J_\nu:=\omega_{\nu,1,1}$ and $I_{\nu}:=\omega_{\nu,1,-1}$. Then $J_\nu$ and $I_{\nu}$ are the Bessel function of the first kind of order $\nu$ and the modified Bessel function of the first kind of order $\nu$ respectively. In fact, they are the solutions of the Bessel's equation $z^2\omega''(z)+z\omega'(z)+(z^2-\nu^2)\omega(z)=0$ and modified Bessel equation $z^2\omega''(z)+z\omega'(z)-(z^2+\nu^2)\omega(z)=0$ having the form
\begin{equation}\label{eqj}
{J}_{\nu}(z)=\sum_{n\geq 0} \frac{(-1)^n}{n! \Gamma (\nu+n+1)}\left(\frac{z}{2}\right)^{2n+\nu}\quad(z,\nu\in\mathbb{C})
\end{equation}
and
\begin{equation}\label{eqi}
I_{\nu}(z)=\sum_{n\geq 0} \frac{1}{n! \Gamma (\nu+n+1)}\left(\frac{z}{2}\right)^{2n+\nu}\qquad(z,\nu\in\mathbb{C})
\end{equation}
respectively. Similarly, the functions  $j_{\nu}:= \sqrt{\pi}\omega_{\nu,2,1}/2$ and $i_{\nu}:= \sqrt{\pi}\omega_{\nu,2,-1}/2$ denote the spherical Bessel function of the first kind of order $\nu$ and the modified spherical Bessel function of the first kind of order $\nu$  having the form
\begin{equation}\label{kv}
{j}_{\nu}(z)=\sqrt{\dfrac{\pi}{2z}}J_{\nu+1/2}(z)=\dfrac{\sqrt{\pi}}{2}\sum_{n\geq 0} \frac{(-1)^n}{n! \Gamma (\nu+n+3/2)}\left(\frac{z}{2}\right)^{2n+\nu}\qquad(z,\nu\in\mathbb{C})
\end{equation}
and
\begin{equation}\label{yv}
{i}_{\nu}(z)=\sqrt{\dfrac{\pi}{2z}}I_{\nu+1/2}(z)=\dfrac{\sqrt{\pi}}{2}\sum_{n\geq 0} \frac{1}{n! \Gamma (\nu+n+3/2)}\left(\frac{z}{2}\right)^{2n+\nu}\qquad(z,\nu\in\mathbb{C})
\end{equation}
respectively. Therefore the study of the geometric properties of $\omega_{\nu,b,c}$  such as univalence, starlikeness and convexity in the unit disk gives a unified treatment to the study of Bessel, modified Bessel, spherical Bessel and modified spherical Bessel functions. For more information on Bessel functions, see \cite{MR2656410,MR2743533,MR2429033,MR3352679,MR0111846,MR3252850,MR123456,MR3568675,MR2486953,MR2826152,MR3743004} and references therein.

Let $\mathcal{A}$ be the class of all analytic functions $f$ in $\mathbb{D}:=\{z\in\mathbb{C}\colon|z|<1 \}$ normalized by the condition $f(0)=0=f'(0)-1$ and $\mathcal{S}$ be its subclass consisting of univalent functions. The notion of subordination is quite useful in describing the containment relationship between the image domains of two analytic functions. For any two analytic functions $f$ and $g$ defined in $\mathbb{D}$, we say that $f$ is subordinate to $g$, written as $f\prec g$, if there exists a Schwarz function $s$ with $s(0)=0$ and $|s(z)|<1$ that satisfies $f(z)=g(s(z))$ for all $z\in\mathbb{D}$. It is easy to see that if $f\prec g$, then $f(0)=g(0)$ and $f(\mathbb{D})\subseteq g(\mathbb{D})$. In particular, if $g$ is univalent in $\mathbb{D}$, then $f\prec g$ if and only if $f(0)=g(0)$ and $f(\mathbb{D})\subseteq g(\mathbb{D})$.  Consider the family $ \mathcal{P}_e$ consisting of  functions $p$ that are analytic in $\mathbb{D}$ with  $p(0)=1$  and $p(z)\prec e^z$. Since $\re(e^z)>0$ for all $z\in\mathbb{D}$, a function $p\in\mathcal{P}_e$ also satisfies $\re(p(z))>0$ in $\mathbb{D}$ and therefore $p$ is also a Carath\'eodory function.

Since the normalization plays a vital role in studying the properties of univalent functions, therefore, it is imperative to normalize the generalized Bessel function $\omega_{\nu,b,c}$ given by \eqref{e1}. By means of the transformation \begin{equation}\label{phi}\varphi_{\nu,b,c}(z)= 2^\nu \Gamma\left(\nu+\frac{b+1}{2}\right)z^{-\nu/2} \omega_{\nu,b,c}(\sqrt{z})\end{equation} and using the Pochhammer symbol  $(x)_\mu$ given by \begin{equation*} (x)_\mu=\frac{\Gamma(x+\mu)}{\Gamma (x)}= \begin{cases}
1, & \mu=0\\
x(x+1)\cdots (x+n-1), &\mu=n\in\mathbb{N}
\end{cases}\end{equation*} 	  the function $\varphi_{\nu,b,c}$ has the infinite series representation\begin{equation}\label{e2}
\varphi_{\nu,b,c}(z)=  \sum_{n\geq 0} \frac{(-c/4)^n}{(\kappa)_n}\frac{z^n}{n!}
\end{equation}
where $\nu$, $b$, $c \in\mathbb{C}$ such that $\kappa=\nu+(b+1)/2\notin\{0,-1,-2,\ldots\}$ and is called the generalized and normalized Bessel function of the first kind of order $\nu$.

Note that $z\varphi_{\nu,b,c}\in\mathcal{A}$. If we set $b_n:=(-c/4)^n/(n!(\kappa)_n )$, then the function $\varphi_{\nu,b,c}\notin\mathcal{A}$ as it is not normalized but $(\varphi_{\nu,b,c}(z)-b_0)/b_1\in\mathcal{A}$. For brevity, we shall denote $\varphi_{\nu,b,c}$ by $\varphi_{\nu}$ and $z\varphi_{\nu,b,c}$ is denoted by  $\vartheta_{\nu,b,c}$ or simply  by $\vartheta_{\nu}$.   Also, note that the function $\varphi_{\nu}$ is entire as the radius of convergence of the  series \eqref{e2} is infinity and it also satisfies the second order differential equation \begin{equation}\label{e4}
4z^2 \varphi''_\nu(z)+4\kappa z \varphi'_\nu(z)+cz\varphi_\nu(z)=0.\end{equation}
Moreover, it satisfies the recurrence relation  \cite[Lemma 1.2, p.~14]{MR2656410}
\begin{equation}\label{e3}
4\kappa \varphi'_\nu(z) = -c\varphi_{\nu+1}(z).
\end{equation}

Determining the sufficient conditions on the parameters $\nu$, $b$, $c$ of generalized and normalized Bessel functions to belong to  well-known classes of univalent functions have a long history. Baricz and Ponnusamy \cite{MR2743533} determined the conditions on the parameters of normalized form of generalized Bessel functions to be convex and starlike in $\mathbb{D}$. Using the theory of differential subordination,  Bohra and Ravichandran \cite{MR3738359} determined the conditions so that $\varphi_{\nu,b,c}$ is strongly convex of order $1/2$. Recently, Madaan \emph{et al.} \cite{MR4048763} determined the lemniscate convexity and other properties of the generalized and normalized Bessel function using the theory of differential subordination. Kanas \textit{et al.\@}  \cite{MR123456} and Mondal and Dhuain \cite{MR3568675} independently  obtained the conditions on the parameters for which $\varphi_{\nu,b,c}$ is Janowski convex and $z\varphi'_{\nu,b,c}$ is Janowski starlike in the unit disk. In Section 2, we determine the conditions on the parameters $\kappa$ and $c$ so that the generalized and normalized Bessel function belongs to the classes $\ke$ and $\se$ introduced and studied by Mendiratta \emph{et al.} \cite{MR3394060}. Here $\ke$ and $\se$ are the subclasses of $\mathcal{S}$ consisting of functions $f$ for which the quantity $w=1+zf''/f'$ (and $w=zf'/f$ respectively) lies in the image domain $|\log w|<1$ under the exponential function $w=e^z$. In terms of subordination, $f\in \se$ if and only if $zf'/f \in \mathcal{P}_e$ and $f\in \ke$ if and only if $1+zf''/f' \in \mathcal{P}_e$. Interesting examples involving trigonometric and hyperbolic functions are provided to illustrate our results and its connection with Libera operator is also established.

Using the technique of convolution to define a linear operator involving generalized hypergeometric functions by Dziok and Srivastava \cite{MR1686354}, recently Baricz \textit{et al.\@}  \cite{MR3352679}  introduced the $B_\kappa^c$-operator  defined by
\begin{equation*}
B_\kappa^c f(z):=(\vartheta_{\nu}\ast f)(z)=z+\sum_{n=1}^{\infty} \frac{(-c/4)^na_{n+1}}{(\kappa)_n}\frac{z^{n+1}}{n!}\qquad(z\in\mathbb{D})
\end{equation*} where $f(z)=z+\sum_{n=1}^{\infty}a_{n+1}z^{n+1}\in\mathcal{A}$. Note that $B_\kappa^c$-operator satisfies the recurrence relation
\begin{equation}\label{b1}
z(B_{\kappa+1}^cf(z))'=\kappa B_\kappa^c f(z)-(\kappa-1)B_{\kappa+1}^cf(z)\qquad(z\in\mathbb{D}).
\end{equation}
Applying the methodology of admissible functions given by Miller and Mocanu \cite{MR1760285}, Baricz \textit{et al.\@} \cite{MR3352679} proved various differential subordination results involving the $B_\kappa^c$-operator. Motivated by their work, we derive a class of admissible functions involving the $B_\kappa^c$-operator in Section \ref{bkc}. Sufficient conditions are also determined so that the function $B_\kappa^cf$ belongs to the classes $\ke$ and $\se$.

We will make use of the theory of differential subordination which was introduced by Miller and Mocanu \cite{MR1760285} in proving the results of Sections \ref{exp} and \ref{bkc}. Let  $\mathcal{Q}$ denote the set of functions $q$ that are analytic and univalent  on $\overline{\mathbb{D}}\setminus \textbf{E}(q)$ where $\textbf{E}(q)=\{\zeta \in \partial \mathbb{D} \colon \lim_{z\to\zeta} q(z)=\infty \}$ and $q'(\zeta)\not=0$ for $\zeta\in\partial\mathbb{D}\setminus \textbf{E}(q)$. Further, we denote the subclass of $\mathcal{Q}$ for which $q(0)=a$ by $\mathcal{Q}_a$. The following definition of admissible functions and foundation result of differential subordination theory will be needed in our study.
\begin{definition}\cite[Definition 2.3(a), p.~ 27]{MR1760285}\label{def2}
	Let $\Omega$ be any subset in $\mathbb{C}$ and $q\in \mathcal{Q}$. Define $\Psi(\Omega,q)$ to be a \emph{class of admissible functions} $\psi\colon \mathbb{C}^3\times \mathbb{D} \to\mathbb{C}$ that satisfies the admissibility condition $\psi(r,s,t;z) \notin \Omega$ whenever $r=q(\zeta)$, $s= m\zeta q'(\zeta)$ and $\re(1+ t/s) \geq m \re( 1+\zeta q''(\zeta)/q'(\zeta))$ where $z\in\mathbb{D}$, $\zeta \in \partial \mathbb{D}\setminus \textbf{E}(q)$ and $m \geq 1$ is a positive integer.
\end{definition}
\begin{theorem}\cite[Theorem 2.3(b), p.~28]{MR1760285}\label{thm1}
	Let $\psi\in\Psi(\Omega,q)$ with $q\in\mathcal{Q}_a$. If an analytic function  $p(z)=a+a_nz^n+a_{n+1}z^{n+1}+\cdots$ satisfies $\psi(p(z),zp'(z),z^2p''(z);$ $z)\in\Omega$, then $p(z)\prec q(z)$.
\end{theorem}

\section{Exponential Convexity and Starlikeness of the Generalized and Normalized Bessel Function}\label{exp}

Using the theory of differential subordination, Naz \textit{et al.\@} \cite{MR3962536}  investigated the class of admissible functions associated with the exponential function $e^z$ and proved the following lemma.
\begin{lemma}\cite{MR3962536}\label{lemA}
	Let $\Omega$ be a subset of $\mathbb{C}$ and the function $\psi \colon \mathbb{C}^3\times \mathbb{D}\to\mathbb{C}$ satisfies the admissibility condition $ \psi(r,s,t;z) \notin \Omega$ whenever
	\begin{equation*}
	r = e^{e^{i\theta}},\quad  s = m e^{i\theta} r \quad
	\text{and} \quad \re (1+t/s) \geq m (1+\cos \theta)
	\end{equation*}
	where $z\in\mathbb{D}$, $\theta \in [0,2\pi)$ and $m \geq 1$. If $p$ is an analytic function in $\mathbb{D}$ with $p(0)=1$ and $\psi(p(z), zp'(z), z^2 p''(z); z) \in \Omega$ for $z\in \mathbb{D}$, then  the function $p \in \mathcal{P}_e$.
\end{lemma}
Observe that the admissibility condition $\psi(r,s,t;z)\not\in\Omega$ is true for all $r=e^{e^{i\theta}}$, $s=me^{i\theta}e^{e^{i\theta}}$ and $t$ such that $\re(1+t/s)\geq 0$, that is,  $\re ((t+s)e^{-i\theta}e^{-e^{i\theta}}) \geq 0$ for all $\theta\in[0,2\pi)$ and $m\geq1$. We ought to employ this form of admissibility condition to demonstrate our results. Furthermore, for the case $\psi \colon  \mathbb{C}^2\times \mathbb{D}\to\mathbb{C}$, the admissibility condition reduces to
\[\psi(e^{e^{i\theta}}, m e^{i\theta} e^{e^{i\theta}};z) \notin \Omega \]
where $z\in\mathbb{D}$,  $\theta \in [0,2\pi)$ and $m \geq 1$.

We first determine the condition on the parameters $\kappa$ and $c$ so that the generalized and normalized Bessel function belongs to the class $\mathcal{P}_e$ using the recurrence relation \eqref{e3}.
\begin{theorem}\label{thmB}
	If the parameters $\kappa$, $c \in\mathbb{C}$  are constrained such that $\re(\kappa) \geq (|c|/4)+1$, then the function $\varphi_\nu\in \mathcal{P}_e$.
\end{theorem}
\begin{proof}
	If $c=0$, then $\varphi_\nu\equiv 1$ which clearly belongs to $\mathcal{P}_e$. Suppose that $c\neq 0$. In order to prove our result, we first claim that the function $(-4\kappa/c)\varphi'_\nu\in\mathcal{P}_e$ if $\re(\kappa)\geq|c|/4$. Let us
	define a function $p\colon\mathbb{D}\to\mathbb{C}$ by \[p(z)=-\frac{4\kappa }{c}\varphi'_\nu(z). \] Then   $p$ is analytic in $\mathbb{D}$ and $p(0)=1$. Suppose that $z\not=0$. Since $\varphi_\nu$ satisfies the differential equation \eqref{e4}, we have \[4z \varphi''_\nu(z)+4\kappa \varphi'_\nu (z)+c \varphi_\nu (z)=0\qquad(z\in\mathbb{D}). \] Differentiation gives $4z\varphi'''_\nu(z)+4(\kappa+1)\varphi''_\nu(z)+c\varphi'_\nu(z)=0$ so that the  function $p$ satisfies
	\begin{equation}\label{e5}
	4z^2 p''(z)+4(\kappa+1) z p'(z)+czp(z)=0\qquad(z\in\mathbb{D}).\end{equation}
	Note that the above equation is also valid for $z=0$. Suppose $\re(\kappa )\geq |c|/4$ and define a function $\psi:\mathbb{C}^3\times \mathbb{D}\to \mathbb{C}$ by $\psi(r,s,t;z)=4t+4(\kappa+1)s+czr$ and let $\Omega:=\{0\}$, then \eqref{e5} can be written as \[\psi(p(z),zp'(z),z^2p''(z);z)\in\Omega\qquad(z\in\mathbb{D}). \]
	To prove the required result, we will make use of Lemma \ref{lemA} which will be applicable if we show that  $\psi(r,s,t;z)\notin\Omega$ whenever $r = e^{e^{i\theta}}$,   $s = m e^{i\theta} e^{e^{i\theta}}$ and $\re((s+t)e^{-i\theta}e^{-e^{i\theta}}) \geq 0 $ where $z\in\mathbb{D}$,  $\theta \in [0,2\pi)$ and $m \geq 1$. For any $z_1$, $z_2\in\mathbb{C}$, recall the  inequality $|z_1+z_2|\geq ||z_1|-|z_2||$  and note that
	\begin{align*}
	|\psi(r,s,t;z)| & = |4(t+s)+4\kappa me^{i\theta}e^{e^{i\theta}}+ cze^{e^{i\theta}}|\\
	&> 4e^{\cos\theta}\left(|(t+s)e^{-i\theta}e^{-e^{i\theta}}+\kappa m |-\frac{1}{4}|c|\right)\\
	& \geq 4 e^{\cos\theta}\left(\re\big((t+s)e^{-i\theta}e^{-e^{i\theta}}\big)+m \re(\kappa)  -\frac{1}{4}|c|\right)\geq0.
	\end{align*} Thus $|\psi(r,s,t;z)|\ne 0$ and using  Lemma \ref{lemA}, we conclude that $p(z)\prec e^z$ for all $z\in\mathbb{D}$, that is, if $c\ne0$ and $\re(\kappa)\geq|c|/4$, then the function $(-4\kappa/c)\varphi'_\nu\in\mathcal{P}_e$.
	Using the recurrence relation  \eqref{e3}, we have \[4(\kappa-1)\varphi'_{\nu-1}(z)=-c\varphi_\nu(z). \]
	Hence it follows that if $c\ne0$ and $\re(\kappa)\geq(|c|/4)+1$, then the function $\varphi_\nu\in\mathcal{P}_e$. This completes the proof of the theorem.
\end{proof}
It is interesting to note that generalized Bessel functions of first kind can be reduced to elementary trigonometric and hyperbolic functions. Therefore the above deduced result is significant and leads to various  interesting relations involving the generalized and normalized Bessel function and  trigonometric functions by selecting the suitable choices of the parameters involved. Let us illustrate Theorem \ref{thmB} by the following example.
\begin{example}
	Clearly, the following functions
	\begin{gather*}
	\varphi_{1}(z)=\varphi_{1,0,2}(z)=\frac{\sin(\sqrt{2z})}{\sqrt{2z}}\\
	\varphi_{2}(z)=\varphi_{2,0,6}(z)=\frac{1}{12}\left(\frac{\sqrt{6}\sin(\sqrt{6z})}{z^{3/2}}-\frac{6\cos(\sqrt{6z})}{z}\right)\\
	\varphi_{3}(z)=\varphi_{3,2,10}(z)=\frac{1}{40}\left(\frac{21(-3+2z)\cos(\sqrt{10z})}{z^{3}}+\frac{63(1-4z)\sin(\sqrt{10z})}{\sqrt{10}z^{7/2}}\right)
	\end{gather*}
	satisfy the hypothesis of Theorem \ref{thmB}. Therefore $\varphi_i(z)\prec e^z$ (for $i=1,2,3)$. These three subordinations are illustrated graphically in Figure \ref{fig:1}.
	
	\begin{figure}[h]
		\begin{center}
			\subfigure[$c=2,\kappa=3/2$]{\includegraphics[width=1.95in]{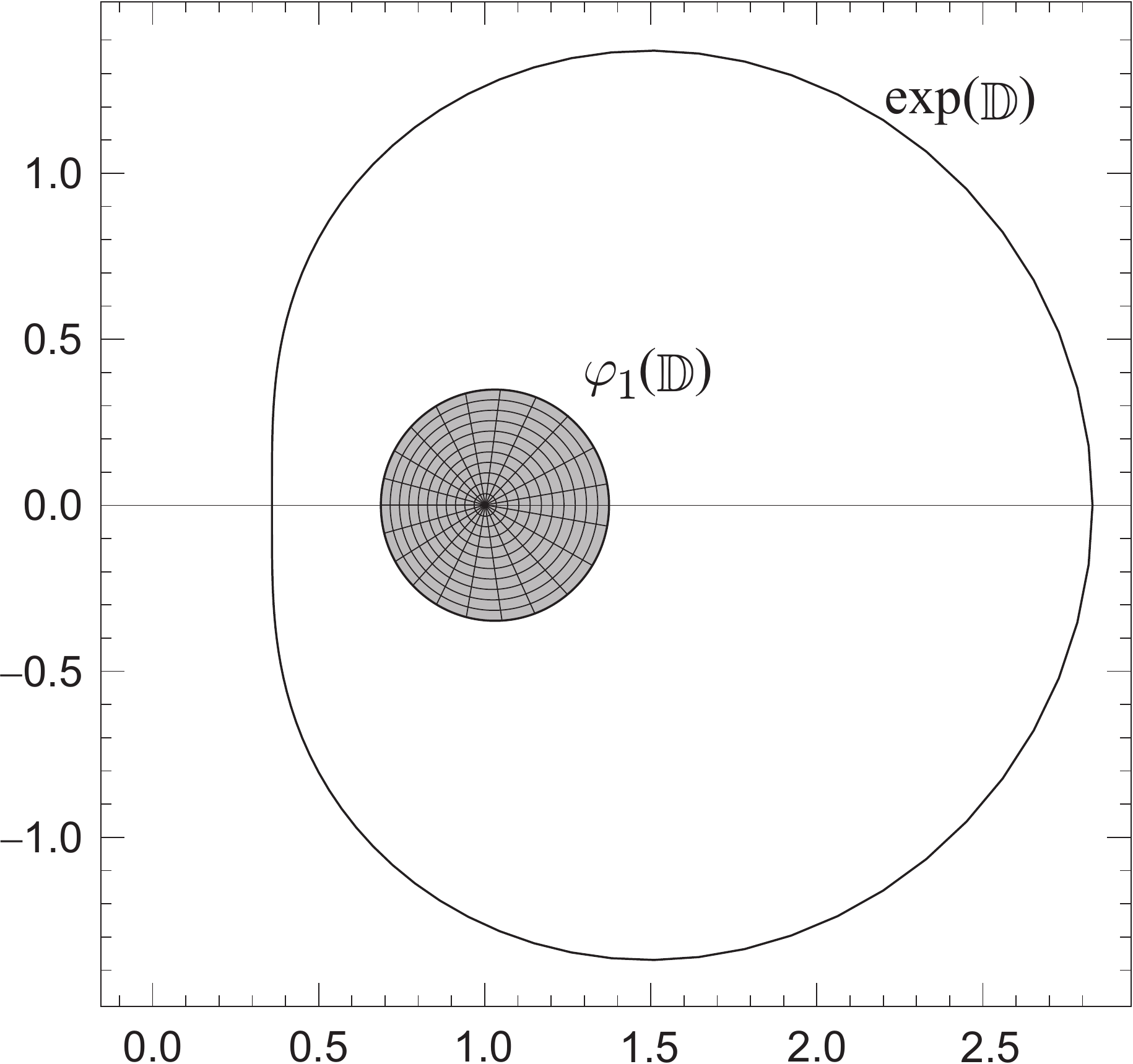}}\hspace{10pt}
			\subfigure[$c=6,\kappa=5/2$]{\includegraphics[width=1.95in]{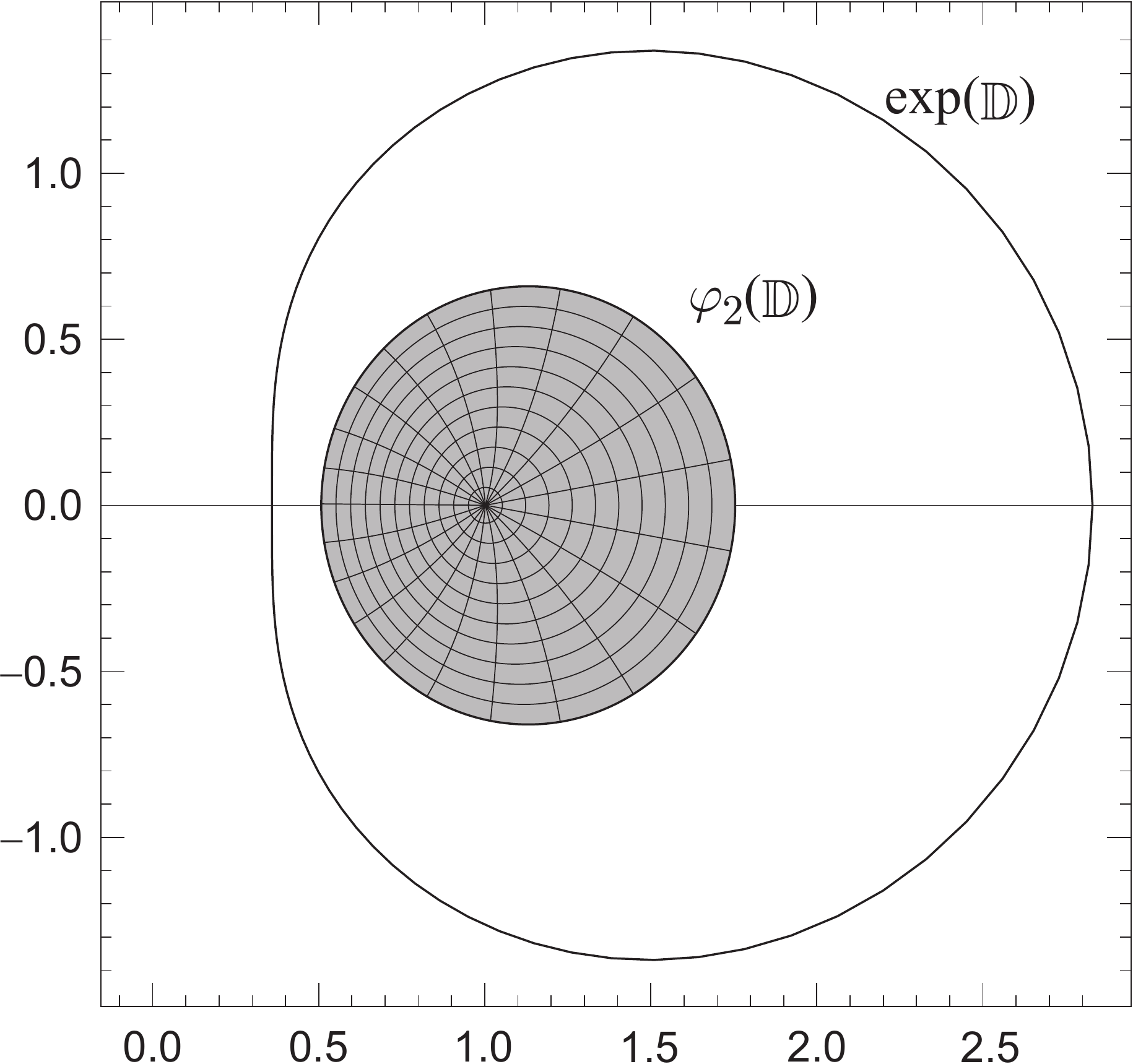}}\hspace{10pt}
			\subfigure[$c=10,\kappa=9/2$]{\includegraphics[width=1.95in]{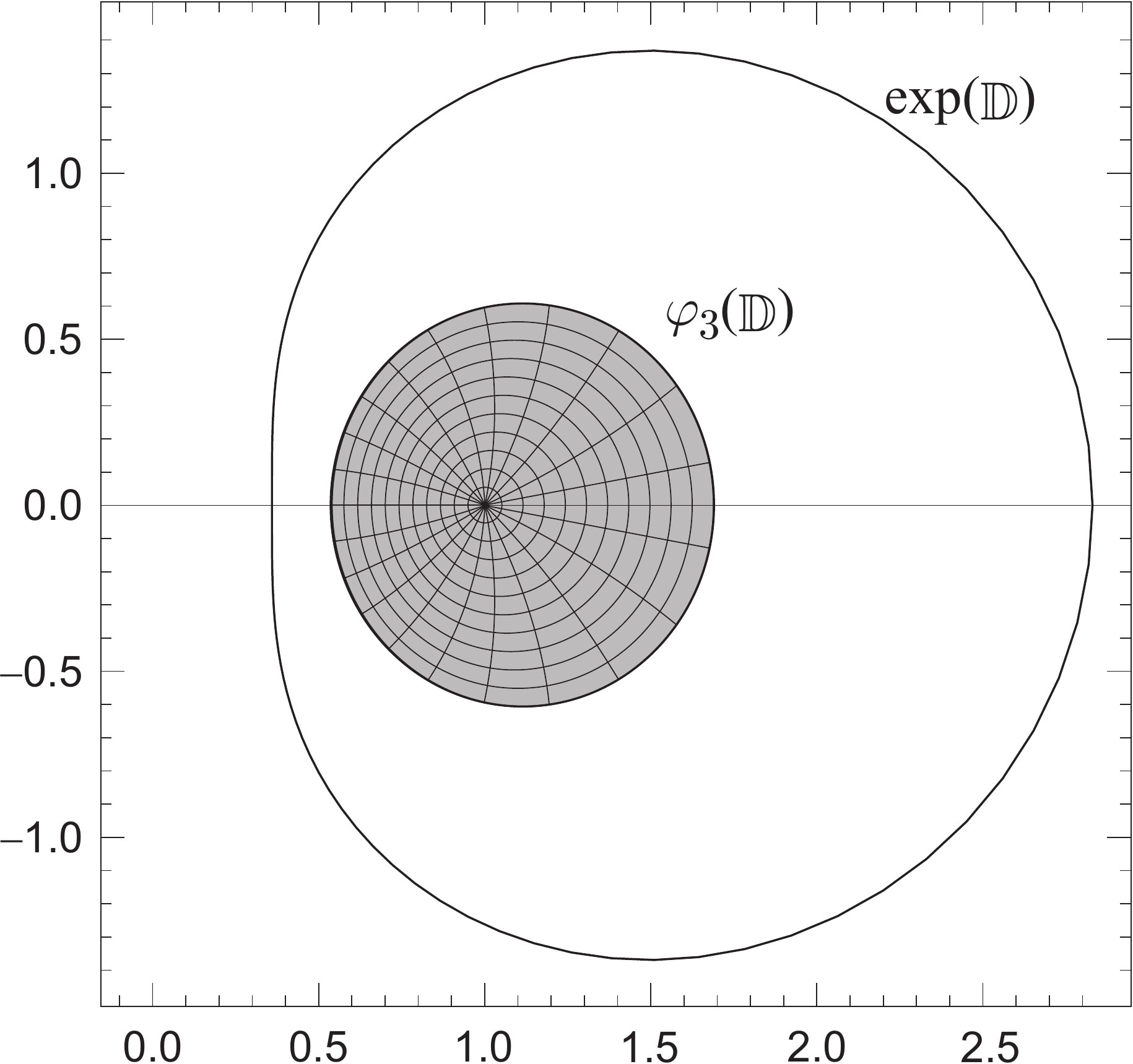}}
			\caption{Graph showing $\varphi_i(\mathbb{D})\subset \exp(\mathbb{D})$ ($i=1,2,3$).}\label{fig:1}
		\end{center}
	\end{figure}
	Similarly, the subordinations $\varphi_8(z)=\varphi_{8, 0,30}(z)\prec e^z$ and $\varphi_{31/2}(z)= \varphi_{31/2, 0,60}$ $(z)\prec e^z$ can be interpreted graphically as shown in Figure \ref{fig:2}.
	\begin{figure}[h]
		\begin{center}
			\subfigure[$c=30,\kappa=17/2$]{\includegraphics[width=1.95in]{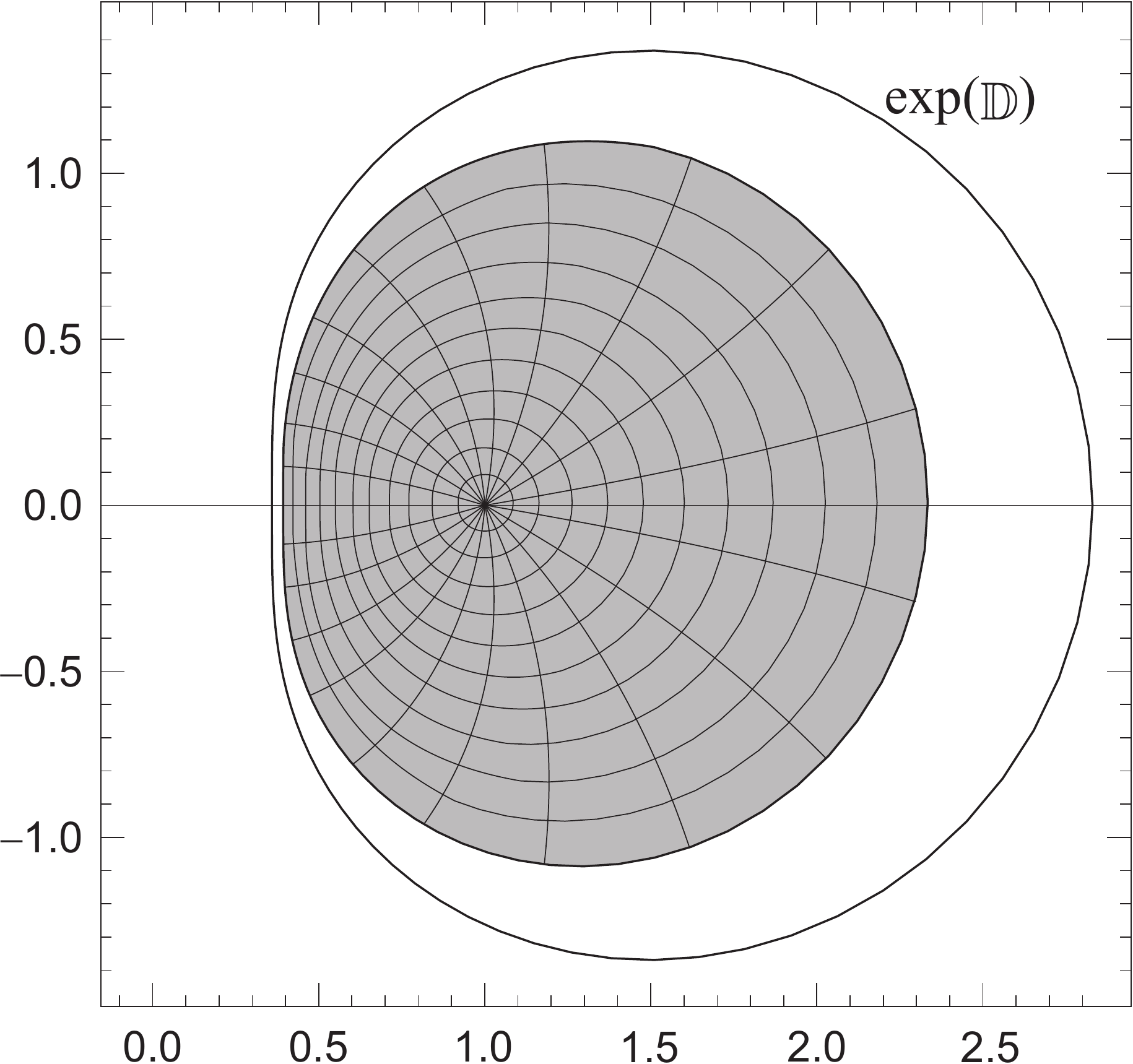}}\hspace{10pt}
			\subfigure[$c=60,\kappa=16$]{\includegraphics[width=1.95in]{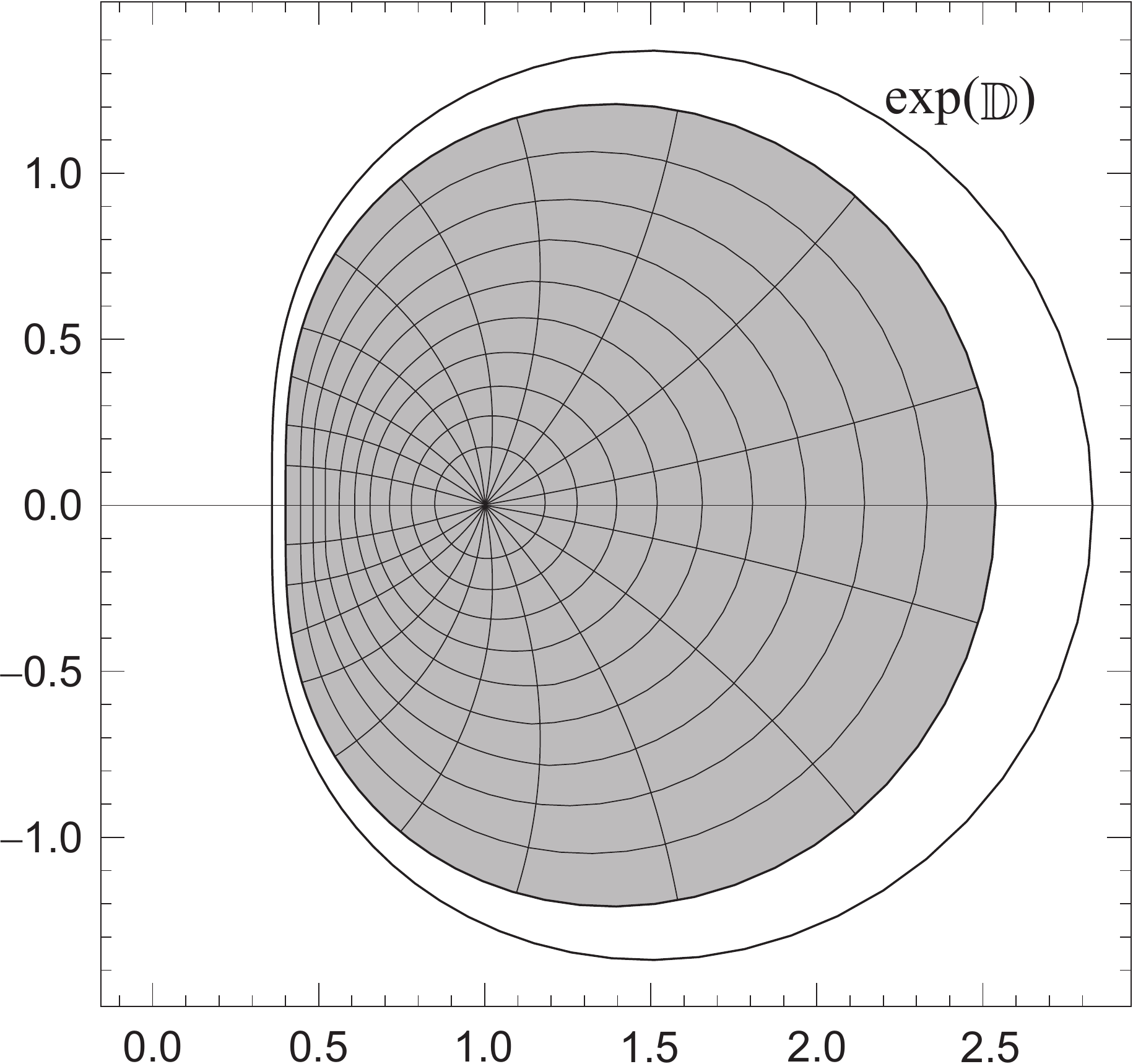}}\hspace{10pt}
			\caption{Graph showing $\varphi_i(\mathbb{D})\subset \exp(\mathbb{D})$ ($i=8,31/2$).}
			\label{fig:2}
		\end{center}
	\end{figure}
\end{example}

The next result deals with a sufficient condition on the parameters $\kappa$ and $c$ so that the generalized and normalized Bessel function belongs to the class $\ke$.

\begin{theorem}\label{thmA}
	Let the parameters $\kappa$, $c \in\mathbb{C}$  be constrained such that $c\neq 0$, $\re(\kappa) \geq|c|/4$  and \begin{equation}\label{e12}
	| \kappa-2| + \frac{1}{4(e-1)}|c|\leq\frac{e^2+e-1}{e^2(e-1)}
	\end{equation} then the function  $-4\kappa(\varphi_\nu-1)/c\in\ke$.
\end{theorem}
\begin{proof}
	Since $\re(\kappa) \geq|c|/4$ and $c\neq 0$, $\varphi_\nu$ is univalent in $\mathbb{D}$ by means of  \cite[Theorem 2.3, p.~29]{MR2656410}  and therefore $\varphi'_\nu(z)\not=0$ for all $z\in\mathbb{D}$. Define a function $p\colon\mathbb{D}\to\mathbb{C}$ by \[p(z)=1+\frac{z(-4\kappa(\varphi_\nu-1)/c)''}{(-4\kappa(\varphi_\nu-1)/c)'}=1+\frac{z\varphi''_\nu(z)}{\varphi'_\nu(z)}.  \]
	Then $p$ is analytic in $\mathbb{D}$ with $p(0)=1$.
	For $z\not=0$,  the differential equation \eqref{e4}  yields $4z\varphi'''_\nu(z)+4(\kappa+1)\varphi''_\nu(z)+c\varphi'_\nu(z)=0$.
	As $\varphi'_\nu(z)\not=0$ for all $z\in\mathbb{D}$,  dividing the previous equation by  $\varphi'_\nu(z)$ and then multiplying it by $z$, we obtain
	\[ 4 \left(\frac{z^2\varphi'''_\nu(z)}{\varphi'_\nu(z)}\right)+4(\kappa+1)\left(\frac{z\varphi''_\nu(z)}{\varphi'_\nu(z)} \right)+cz=0. \]
	By making use of the equation $z\varphi''_\nu(z)/\varphi'_\nu(z)=p(z)-1$, a straightforward calculation shows that
	\[\frac{z^2\varphi'''_\nu(z)}{\varphi'_\nu(z)} = zp'(z)+p^2(z)-3p(z)+2.\] Hence the function $p$ satisfies the differential equation
	\begin{equation*}\label{e11}
	4z p'(z)+4p^2(z)+4(\kappa-2)(p(z)-1) +cz-4=0\end{equation*}
	which is also true for $z=0$. Define a function $\psi(r,s;z):=4s+4r^2+4(\kappa-2)(r-1)+cz-4$ and let $\Omega:=\{0\}$. Then $\psi(p(z),zp'(z);z)\in\Omega$ for all $z\in\mathbb{D}$.  Again, we will employ Lemma \ref{lemA} to prove  that $p(z)\prec e^z$ for all $z\in\mathbb{D}$,  that is, we will show that $\psi(r,s;z)\notin \Omega$ where $r = e^{e^{i\theta}}$ and   $s = m e^{i\theta} e^{e^{i\theta}}$  for $z\in\mathbb{D}$,  $\theta \in [0,2\pi)$ and $m \geq 1$ under the given condition \eqref{e12}.
	Note that
	\begin{align*}
	|\psi(r,s;z)|&=4\left|s+r^2-1+(\kappa-2)(r-1)+\frac{1}{4}cz \right|\\
	&> 4 \left(|s+r^2-1| - |\kappa-2||r-1| -\frac{1}{4}|c| \right).
	\end{align*}
	For $r = e^{e^{i\theta}}$, $s = m e^{i\theta} e^{e^{i\theta}}$, $\theta\in[0,2\pi)$, $z\in \mathbb{D}$ and $m \geq 1$, we get
	\begin{align*}|s+r^2-1|^2& =(me^{\cos\theta}\cos(\theta+\sin\theta)+e^{2\cos\theta}\cos(2\sin\theta)-1)^2\\
	&\quad\;+(me^{\cos\theta}\sin(\theta+\sin\theta)+e^{2\cos\theta}\sin(2\sin\theta))^2:=g_1(\theta).
	\end{align*}
	Using  the second derivative test, the function $g_1$ attains its minimum value at $\theta=\pi$ so that
	\[\min_{\theta\in[0,2\pi)}g_1(\theta)=g_1(\pi)=\left(\frac{-m}{e}+\frac{1}{e^2}-1\right)^2\]
	which implies
	\[|s+r^2-1|\geq \frac{m}{e}-\frac{1}{e^2}+1 \geq  \frac{1}{e}-\frac{1}{e^2}+1.\]
	Similarly, we have
	\begin{align*}
	|r-1|^2&=(e^{\cos\theta}\cos(\sin\theta)-1)^2+e^{2\cos\theta}\sin^2(\sin\theta)\\
	&=1+e^{2\cos\theta}-2e^{\cos\theta}\cos(\sin\theta):=g_2(\theta).
	\end{align*}
	The function $g_2$ attains its maximum value at $\theta=0$ and thus $|r-1|\leq e-1$. Using these calculations, we obtain
	\begin{align*}
	|\psi(r,s;z)|=|\psi(e^{e^{i\theta}},me^{i\theta} e^{e^{i\theta}};z)|
	> 4\left(\frac{1}{e}-\frac{1}{e^2}+1 -(e-1)|\kappa-2|-\frac{1}{4}|c| \right)\geq 0
	\end{align*} using \eqref{e12}. Therefore $|\psi(r,s;z)|>0$ and hence Lemma \ref{lemA} completes the proof.
\end{proof}
Note that if \begin{equation}\label{eq}
| \kappa-3|  + \frac{1}{4(e-1)}|c|\leq\frac{e^2+e-1}{e^2(e-1)}
\end{equation}  then the function $-4(\kappa-1)(\varphi_{\nu-1}-1)/c\in\ke$ by Theorem  \ref{thmA} provided that $c\neq 0$ and $\re(\kappa) \geq(|c|/4)+1$. Using the famous Alexander duality theorem between the classes $\ke$ and $\se$, that is, $f\in\ke$ if and only if $zf'\in\se$, we obtain $-4(\kappa-1)z\varphi_{\nu-1}'/c\in\se$. Also the recurrence relation \eqref{e3} gives $cz\varphi_{\nu}(z)=-4(\kappa-1)z\varphi'_{\nu-1}(z)$. Consequently, we have the following result.

\begin{theorem}\label{cor1}
	Let the parameters $\kappa$, $c \in\mathbb{C}$  be constrained such that  $c\neq 0$, $\re(\kappa) \geq(|c|/4)+1$ and \eqref{eq} is satisfied, then the function  $\vartheta_\nu\in\se$ where $\vartheta_\nu(z)=z\varphi_\nu(z)$ for all $z\in\mathbb{D}$.
\end{theorem}
The particular choices of $b$ and $c$ in Theorems \ref{thmA} and \ref{cor1} leads to the corresponding results for Bessel ($b=c=1$), modified Bessel ($b=1$ and $c=-1$), spherical Bessel ($b=2$ and $c=1$) and modified spherical Bessel ($b=2$ and $c=-1$) functions. For Bessel and modified Bessel function, we have $|c|=1$ and $\kappa=\nu+1$ so that the following corollary is obtained.
\begin{corollary}\label{cora}
	Let the parameter $\nu\in\mathbb{C}$.  For the functions
	\[\mathcal{J}_\nu(z):=2^\nu \Gamma(\nu+1)z^{-\nu/2} J_\nu(\sqrt{z})\]
	and
	\[\mathcal{I}_\nu(z):=2^\nu \Gamma(\nu+1)z^{-\nu/2} I_\nu(\sqrt{z})\]
	where $J_\nu$ and $I_\nu$ are the Bessel function and the modified Bessel function of the first kind of order $\nu$ defined by \eqref{eqj} and \eqref{eqi} respectively, the following assertions hold.
	\begin{enumerate}[(a)]
		\item If $\re(\nu)\geq-0.75$ and
		\[|\nu-1|\leq\frac{1}{e^2}+\frac{3}{4(e-1)} \] then the functions $-4(\nu+1)(\mathcal{J}_\nu-1)\in\ke$ and $-4(\nu+1)(\mathcal{I}_\nu-1)\in\ke$.
		\item If $\re(\nu)\geq0.25$ and \[|\nu-2|\leq\frac{1}{e^2}+\frac{3}{4(e-1)} \] then the functions $z\mathcal{J}_\nu\in\se$ and $z\mathcal{I}_\nu\in\se$.
	\end{enumerate}
\end{corollary}
To illustrate Corollary \ref{cora}, consider the functions
\[\mathcal{J}_{1/2}(z)=\sqrt{\frac{\pi}{2}} z^{-1/4}J_{1/2}(\sqrt{z})=\frac{\sin\sqrt{z}}{\sqrt{z}}\]
and
\[\mathcal{I}_{1/2}(z)=\sqrt{\frac{\pi}{2}} z^{-1/4}I_{1/2}(\sqrt{z})=\frac{\sinh\sqrt{z}}{\sqrt{z}}.\]
These functions satisfy the hypothesis of Corollary \ref{cora}(a). Therefore both the functions $-6(\mathcal{J}_{1/2}-1)$ and $-6(\mathcal{I}_{1/2}-1)$ belong to the class $\ke$. In terms of subordination, it can be written as
\[1+\frac{z \mathcal{J}''_{1/2}(z)}{\mathcal{J}'_{1/2}(z)}=\frac{(1-z)\sin\sqrt{z}-\sqrt{z}\cos\sqrt{z}}{2\sqrt{z}\cos\sqrt{z}-2\sin\sqrt{z}} \prec e^z \]and
\[1+\frac{z \mathcal{I}''_{1/2}(z)}{\mathcal{I}'_{1/2}(z)}=\frac{(1+z)\sinh\sqrt{z}-\sqrt{z}\cosh\sqrt{z}}{2\sqrt{z}\cosh\sqrt{z}-2\sinh\sqrt{z}} \prec e^z. \]
Similarly, the functions
\[\mathcal{J}_{3/2}(z)=3\sqrt{\frac{\pi}{2}} z^{-3/4}J_{3/2}(\sqrt{z})\quad \mbox{and}\quad \mathcal{I}_{3/2}(z)=3\sqrt{\frac{\pi}{2}} z^{-3/4}I_{3/2}(\sqrt{z})\]
satisfy Corollary \ref{cora}(b) and therefore the functions
\[z\mathcal{J}_{3/2}(z)=3
\left(\frac{\sin\sqrt{z}}{\sqrt{z}}-\cos\sqrt{z}\right)\quad \mbox{and}\quad z\mathcal{I}_{3/2}(z)=3
\left(\cosh\sqrt{z}-\frac{\sinh\sqrt{z}}{\sqrt{z}}\right) \]
are in $\mathcal{S}^*_e$. Also, the similar reasoning shows that the functions
\[z\mathcal{J}_{5/2}(z)=\frac{15((3-z)\sin\sqrt{z}-3\sqrt{z}\cos \sqrt{z})}{z^{3/2}}\]
and
\[z\mathcal{I}_{5/2}(z)=\frac{15((3+z)\sinh\sqrt{z}-3\sqrt{z}\cos \sqrt{z})}{z^{3/2}}\]
belong to the class $\mathcal{S}^*_e$.

In the similar fashion, if we take $|c|=1$ and $\kappa=\nu+(3/2)$, we can obtain the following result for spherical Bessel and modified spherical Bessel functions.
\begin{corollary}\label{sbf}
	Let the parameter $\nu\in\mathbb{C}$. For the functions
	\begin{equation*}
	\mathfrak{j}_\nu(z):=\frac{1}{\sqrt{\pi}}2^{\nu+1} \Gamma\left(\nu+\frac{3}{2}\right)z^{-\nu/2} j_\nu(\sqrt{z}) \end{equation*}
	and
	\begin{equation*}
	\mathfrak{i}_\nu(z):=\frac{1}{\sqrt{\pi}}2^{\nu+1} \Gamma\left(\nu+\dfrac{3}{2}\right)z^{-\nu/2} i_\nu(\sqrt{z}) \end{equation*}
	where $j_\nu$ and $i_\nu$ are the spherical Bessel function and the modified spherical Bessel function of the first kind of order $\nu$ defined by \eqref{kv} and \eqref{yv} respectively, the following assertions hold.
	\begin{enumerate}[(a)]
		\item If $\re(\nu)\geq-1.25$ and \begin{equation*}|2\nu-1|\leq \dfrac{2}{e^2}+\dfrac{3}{2(e-1)}\end{equation*}
		then the functions $-2(2\nu+3)(\mathfrak{j}_\nu-1)\in\ke$ and $-2(2\nu+3)(\mathfrak{i}_\nu-1)\in\ke$.
		\item If $\re(\nu)\geq-0.25$ and \begin{equation*}|2\nu-3|\leq \dfrac{2}{e^2}+\dfrac{3}{2(e-1)}\end{equation*} then the functions $z\mathfrak{j}_\nu\in\se$ and $z\mathfrak{i}_\nu\in\se$.
	\end{enumerate}
\end{corollary}

It has been proved in \cite{MR3394060} that both $\ke$ and $\se$ are closed under convolution with the convex functions. We will make use of this observation to study the behaviour of an integral operator. The Libera operator $L:\mathcal{A}\to\mathcal{A}$ is defined as
\[ L[f](z):=\frac{2}{z}\int_{0}^{z} f(t) dt=\frac{-2(z+\log(1-z))}{z}\ast f(z) \]
where $f \in \mathcal{A}$ and $z\in \mathbb{D}$. Theorems \ref{thmA} and \ref{cor1} yield the following result which helps in  constructing starlike and convex functions involving $\varphi_\nu$ and $\vartheta_\nu$.

\begin{corollary}
	If the parameters $\nu$, $b$, $c$ are constrained as in Theorem \ref{thmA}, then the function $(-4\kappa(\varphi_\nu-1)/c) \ast f \in\ke$ for every convex function $f\in\mathcal{A}$ and in particular, $L[-4\kappa(\varphi_\nu-1)/c]$ belongs to the class $\ke$. Similarly, if the parameters $\nu$, $b$, $c$ are constrained as in Theorem \ref{cor1}, then the function $\vartheta_\nu \ast f \in\se$ for every convex function $f\in\mathcal{A}$ and in particular, $L[\vartheta_\nu]$ belongs to the class $\se$.
\end{corollary}

Let us make use of this corollary to obtain new functions in the classes $\ke$ and $\se$. By the discussion succeeding Corollary \ref{cora}, it follows that
\[L[-6(\mathcal{J}_{1/2}-1)](z)=\frac{12}{z}(z+2\cos\sqrt{z}-2)\]
\[L[-6(\mathcal{I}_{1/2}-1)](z)=\frac{12}{z}(z-2\cosh\sqrt{z}+2)\]
belong to $\ke$. Similarly, it can be deduced that
\[L[z\mathcal{J}_{3/2}](z)=-\frac{12}{z}(\sqrt{z}\sin\sqrt{z}+2\cos \sqrt{z}-2)\]
\[L[z\mathcal{I}_{3/2}](z)=\frac{12}{z}(\sqrt{z}\sinh\sqrt{z}-2\cosh \sqrt{z}+2)\]
%\[L[z\mathcal{J}_{5/2}](z)=\frac{60}{z}\left(2+\cos\sqrt{z}-\frac{3\sin \sqrt{z}}{\sqrt{z}}\right)\]
%\[I[z\mathcal{J}_{5/2}](z)=\frac{60}{z}\left(2+\cosh\sqrt{z}-\frac{3\sinh \sqrt{z}}{\sqrt{z}}\right)\]
are in the class $\se$.

The last theorem of this section gives a sufficient condition under which the function $2^\nu \Gamma(\kappa) z^{1-\nu}\omega_{\nu,b,c}\in\se$ where $\omega_{\nu,b,c}$ is given by \eqref{e6}. To prove this, we shall be requiring the following result by Miller and Mocanu \cite[Theorem 2.3h, p.~34]{MR1760285} for $M=1/4$ and $a=0$.
\begin{lemma}\label{1/4}
	Let $\Omega$ be a subset of $\mathbb{C}$ and the function $\psi \colon \mathbb{C}^3\times \mathbb{D}\to\mathbb{C}$ satisfies the admissibility condition $ \psi(r,s,t;z) \notin \Omega$ whenever $r = e^{i\theta}/4$, $s = m e^{i\theta}/4$ and $\re(1+t/s) \geq m$ where $z\in\mathbb{D}$, $\theta \in [0,2\pi)$ and $m \geq 1$. If $p$ is an analytic function in $\mathbb{D}$ with $p(0)=0$ and $\psi(p(z), zp'(z), z^2 p''(z); z) \in \Omega$ for $z\in \mathbb{D}$, then   $|p(z)|<1/4$.
\end{lemma}
As noticed earlier, the admissibility condition $\psi(r,s,t;z) \notin \Omega$ in Lemma \ref{1/4} is verified for all $r=e^{i\theta}/4$, $s=me^{i\theta}/4$ and $\re((t+s)e^{-i\theta})\geq1/4$ for all $\theta\in[0,2\pi)$ and $m\geq1$. Using this idea, we prove the following theorem.
\begin{theorem}\label{w}
	If the parameters $\kappa\in\mathbb{R}$, $c \in\mathbb{C}$  are constrained such that \begin{equation}\label{hyp}
	\kappa\geq \max\left\{\dfrac{1}{4}|c|+1,\dfrac{5}{3}|c|+\dfrac{3}{4}\right\}
	\end{equation} then the function $2^\nu \Gamma(\kappa) z^{1-\nu}\omega_{\nu,b,c}\in\se$.
\end{theorem}
\begin{proof}
	As $\kappa \geq |c|/4+1$, $\varphi_\nu(z)\neq 0$ for all $z\in \mathbb{D}$ by Theorem \ref{thmB}. Therefore the function $p\colon\mathbb{D}\to\mathbb{C}$ defined by \[p(z)=\frac{z\varphi'_\nu(z)}{\varphi_\nu(z)} \]
	is analytic in $\mathbb{D}$ with $p(0)=0$. Firstly, we claim that $|p(z)|<1/4$ for all $z\in \mathbb{D}$. Since $p(z)\varphi_\nu(z)=z\varphi'_\nu(z)$ and the function $\varphi_\nu$ satisfies \eqref{e4}, it follows that  $(4zp'(z)+4p^2(z)+4(\kappa-1) p(z)+cz)\varphi_\nu(z)=0$ for all $z\in\mathbb{D}$.  Set $q(z):= 4zp'(z)+4p^2(z)+4(\kappa-1) p(z)+cz$, then the preceding equation can be rewritten as $q(z)\varphi_\nu(z)=0$ for all $z\in\mathbb{D}$. Differentiating this equation and then multiplying it by $z$ gives $\big(zq'(z)+q(z)p(z)\big)\varphi_\nu(z)=0$ for all $z\in\mathbb{D}$. As $\varphi_\nu(z)\ne0$ for all $z\in\mathbb{D}$, we have $zq'(z)+q(z)p(z)=0$ for all $z\in\mathbb{D}$. Hence the function $p$ satisfies the differential equation \[4\big(z^2p''(z)+\kappa zp'(z)+3zp(z)p'(z)+p^3(z)+(\kappa-1)p^2(z)\big)+(p(z)+1)cz=0. \] Rearrangement of the terms yields
	\begin{equation}4\big(z^2p''(z)+zp'(z)+(\kappa-1) (zp'(z)+p^2(z))+3zp(z)p'(z)+p^3(z)\big)+(p(z)+1)cz=0. \label{pz}\end{equation}
	Define a function $\psi(r,s,t;z):\mathbb{C}^3\times\mathbb{D}\to\mathbb{C}$ by \[\psi(r,s,t;z)=4\big(t+s+(\kappa-1)(s+r^2)+3sr+r^3\big)+(r+1)cz\] and suppose that $\Omega:=\{0\}$. Then \eqref{pz} can be written as $\psi(p(z),zp'(z),z^2p''(z);$ $z)\in\Omega$.
	To prove our claim, we shall use Lemma \ref{1/4} which will be apt if we prove that  $\psi(r,s,t;z)\notin\Omega$ whenever $r=e^{i\theta}/4$, $s=me^{i\theta}/4$ and  $\re((t+s)e^{-i\theta})\geq1/4$ for all $\theta\in[0,2\pi)$ and $m\geq1$. Using the given hypothesis \eqref{hyp}, it is easy to see that
	\begin{align*}
	& |\psi(r,s,t;z)| \\ & = \left|4\left(t+s+(\kappa-1)\left(\dfrac{me^{i\theta}}{4}+\dfrac{e^{2i\theta}}{16}\right)+\dfrac{3}{16}me^{2i\theta}+\dfrac{e^{3i\theta}}{64}\right)+\left(\dfrac{e^{i\theta}}{4}+1\right)cz\right|\\
	&> \left|4(t+s)e^{-i\theta}+(\kappa-1)\left(m+\dfrac{e^{i\theta}}{4}\right)+\dfrac{3}{4}me^{i\theta}\right|-\dfrac{1}{16}-\left|\dfrac{e^{i\theta}}{4}+1\right| \cdot |c|\\
	& \geq 4\re\left((t+s)e^{-i\theta}\right)+(\kappa-1)\left(m+\dfrac{\cos\theta}{4}\right)+\dfrac{3}{4}m \cos\theta-\dfrac{1}{16}-\left|\dfrac{e^{i\theta}}{4}+1\right| \cdot |c|\\
	& \geq\dfrac{3}{4}\kappa-\dfrac{9}{16}- \dfrac{5}{4}|c|\geq0.
	\end{align*} 	
	Thus $|\psi(r,s,t;z)|\ne 0$ and using  Lemma \ref{1/4}, we conclude that $|p(z)|<1/4$ for all $z\in\mathbb{D}$, that is
	\[\left|\frac{z\varphi'_\nu(z)}{\varphi_\nu(z)}\right|<\frac{1}{4}\quad (z\in \mathbb{D}).\]
	Now let us suppose $h(z):=2^\nu \Gamma(\kappa)z^{1-\nu}\omega_{\nu,b,c}(z)$.   Using \eqref{phi}, it follows that $h(z)=z\varphi_\nu(z^2)$.  We need to show that $h\in\se$, that is, $|\log zh'(z)/h(z)|<1$ for all $z\in \mathbb{D}$. By making use of the fact that if $|z|<1/2$, then $|\log(1+z)|\leq 3|z|/2$ (see \cite[p.\ 165]{MR1344449}) under the given condition \eqref{hyp}, we have
	\begin{equation*}
	\left|\log\dfrac{zh'(z)}{h(z)}\right|=\left|\log\left(1+\dfrac{2z^2\varphi'_\nu(z^2)}{\varphi_\nu(z^2)}\right)\right|\leq 3 \left|\dfrac{z^2\varphi'_\nu(z^2)}{\varphi_\nu(z^2)} \right|<\dfrac{3}{4}<1.
	\end{equation*}
	This completes the proof of the theorem.
\end{proof}

As a particular case of Theorem \ref{w}, the functions $2^\nu \Gamma(\nu+1) z^{1-\nu}J_{\nu}$ and $2^\nu \Gamma(\nu+1) z^{1-\nu}I_{\nu}$ belongs to $\se$ if $\nu\geq 17/12\approx 1.416$, where $J_\nu$ and $I_\nu$ are the Bessel function and the modified Bessel function of the first kind of order $\nu$ defined by \eqref{eqj} and \eqref{eqi} respectively. For instance, the functions
\[2^{3/2} \Gamma(5/2) z^{-1/2}J_{3/2}=\frac{3(\sin z -z\cos z)}{z^2}\]
and
\[2^{3/2} \Gamma(5/2) z^{-1/2}I_{3/2}=\frac{3(z\cosh z-\sinh z)}{z^2}\]
are in the class $\se$.

\section{Exponential Starlikeness  Associated with the $B_\kappa^c$-operator}\label{bkc}
In this section, we will be investigating the differential subordination results involving the $B_\kappa^c$-operator and its connection with exponential function. Throughout this section, we shall consider the function $\vartheta_{\nu}=z\varphi_{\nu}$ where $\varphi_\nu$ is given by \eqref{e2}.

\begin{theorem}\label{cor4}
	Let a function $f\in\mathcal{A}$  with $\re \kappa\geq \max \{2, |c|/4+(\imag \kappa)^2/6+3/2\}$.
	\begin{itemize}
		\item[(a)] If $f$ is a convex function and $B_{\kappa-1}^c f\in\se$, then $B_\kappa^c f\in\se$.
		\item[(b)] If $zf'$ is a convex function and $B_{\kappa-1}^c f\in\ke$, then $B_\kappa^c f\in\ke$.
	\end{itemize}
\end{theorem}

\begin{proof}
	Note that $\vartheta_\nu$ is a starlike function by \cite[Theorem 2.12(b), p.~ 46]{MR2656410}. For the proof of part (a), since convolution of a starlike function and a convex function is starlike, therefore $B_{\kappa}^cf(z)\neq 0$ for all $0\neq z \in \mathbb{D}$. Let us define a function $p\colon\mathbb{D}\to\mathbb{C}$ by \[p(z)=\frac{z (B_{\kappa}^cf(z))'}{B_{\kappa}^cf(z)}. \]Then $p$ is analytic in $\mathbb{D}$ with $p(0)=1$. Using \eqref{b1}, we have
	\[\dfrac{B_{\kappa-1}^cf(z)}{B_{\kappa}^cf(z)} =\dfrac{1}{\kappa-1}(p(z)+\kappa-2) \]
	which gives
	\[\frac{z (B_{\kappa-1}^cf(z))'}{B_{\kappa-1}^cf(z)}=p(z)+\dfrac{zp'(z)}{p(z)+\kappa-2}. \]
	Since $B_{\kappa-1}^c f\in\se$, we have
	\[p(z)+\dfrac{zp'(z)}{p(z)+\kappa-2}\prec e^z. \]
	As  $\re(e^z)>0$ for all $z\in\mathbb{D}$ and $\re \kappa\geq2$, we must have $\re(e^z+\kappa-2)>0$ for all $z\in \mathbb{D}$. Since $e^z$ is convex univalent in $\mathbb{D}$, we have $p(z)\prec e^z$ using  \cite[Theorem 3.2(a), p.~81]{MR1760285}. Hence it follows that $B_\kappa^c f\in\se$.
	
	For the proof of part (b), let $B_{\kappa-1}^c f\in\ke$. Alexander duality theorem between the classes $\ke$ and $\se$ gives $z(B_{\kappa-1}^c f)'=B_{\kappa-1}^c (zf')\in\se$. By part (a), $B_{\kappa}^c (zf')=z(B_{\kappa}^c  f)'\in\se$ as $zf'$ is a convex function. By again applying Alexander duality theorem, it follows that $B_{\kappa}^c f\in\ke$.
\end{proof}

Let $f(z)=z/(1-z)\in\mathcal{A}$. Then $f$ is a convex function and $B_{\kappa}^cf=\vartheta_\nu$. If $\mathcal{J}_\nu$ and $\mathcal{I}_\nu$ are the generalized and normalized forms of Bessel function and modified Bessel function of the first kind of order $\nu$ respectively as defined in Corollary \ref{cora}, then Theorem \ref{cor4} takes the following form.
\begin{corollary}\label{cor5}
	Let $\nu\in\mathbb{R}$ with $\nu\geq 1$. If $z\mathcal{J}_\nu \in \se$, then $z\mathcal{J}_{\nu+1}\in\se$. Similarly, if $z\mathcal{I}_\nu \in \se$, then $z\mathcal{I}_{\nu+1}\in\se$.
\end{corollary}

Since $z\mathcal{J}_{3/2}$ and $z\mathcal{I}_{3/2}$ are members of $\se$ (see the discussion succeeding Corollary \ref{cora}), therefore Corollary \ref{cor5} shows that $z\mathcal{J}_{(3/2)+n}\in \se$ and $z\mathcal{I}_{(3/2)+n}$ $\in \se$ for all $n\in \mathbb{N}$.

Now, we will define a class $\Phi_C(\Omega,q)$ of admissible functions and use it to prove the corresponding theorem for differential subordination, as done by Miller and Mocanu \cite{MR1760285} for proving Theorem \ref{thm1} using Definition \ref{def2}. Throughout the discussion, it is assumed that the terms appearing in the denominator are non-zero so that all the expressions are well-defined.
\begin{definition}
	Let $\Omega$ be any subset of $\mathbb{C}$ and $q\in\mathcal{Q}_1$. The class of admissible functions $\Phi_{C}(\Omega,q)$ consists of those functions $\phi\colon \mathbb{C}^3\times\mathbb{D}\to\mathbb{C}$ that satisfies the following admissiblity condition	 \[\phi(u,v,w;z) \notin \Omega\] whenever \[u=q(\zeta) \quad\text{and}\quad v	= q(\zeta)+\frac{m\zeta q'(\zeta)}{q(\zeta)}\] and \[  \re\left(\frac{1-2u^2+3uv-2v+(1-v)w}{u-v}\right) \geq m \re\left( 1+\frac{\zeta q''(\zeta)}{q'(\zeta)} \right)\] where $z\in\mathbb{D}$, $\zeta \in \partial \mathbb{D}\setminus \textbf{E}(q)$ and $m \geq 1$ is a positive integer.
\end{definition}
\begin{theorem}\label{thm5}
	Let $\phi\in\Phi_{C}(\Omega,q)$. If a function $f\in\mathcal{A}$  satisfies the following inclusion relation\begin{equation}
	\label{e16}
	\left\{\phi\left(\frac{z(B_{\kappa}^cf(z))'}{B_{\kappa}^cf(z)}, 1+\frac{z(B_{\kappa}^cf(z))''}{(B_{\kappa}^cf(z))'},1+\frac{z(B_{\kappa}^cf(z))'''}{(B_{\kappa}^cf(z))''};z\right)\colon z\in\mathbb{D}\right\}\subset \Omega
	\end{equation}
	then \begin{equation*}
	\frac{z(B_{\kappa}^cf(z))'}{B_{\kappa}^cf(z)}\prec q(z)\qquad(z\in\mathbb{D}).
	\end{equation*}
\end{theorem}
\begin{proof}
	Define an analytic function $p\colon\mathbb{D}\to\mathbb{C}$ by $p(z)=z(B_\kappa^c f(z))'/B_\kappa^c f(z)$. A straightforward calculation gives
	\[1+\frac{z(B_\kappa^cf(z))''}{(B_\kappa^c f(z))'}=p(z)+\frac{zp'(z)}{p(z)}\]
	and	
	\[1+\frac{z(B_\kappa^cf(z))'''}{(B_\kappa^c f(z))''}=\frac{z^2 p''(z)+3 z p(z)p'(z)- z p'(z)  + p^3(z) - 2 p^2(z)+p(z)  }{ zp'(z)+p^2(z)-p(z)}.\]
	Let us define the transformation from $\mathbb{C}^3$ to $\mathbb{C}$ by
	\[u:=r, \quad v :=r+\frac{s}{r}\quad\text{and}\quad w:=\frac{t+3sr-s+r^3-2r^2+r}{s+r^2-r}. \]
	If we write
	\[\psi(r,s,t;z)=\phi(u,v,w;z) =\phi\left(r,r+\frac{s}{r},\frac{t+3sr-s+r^3-2r^2+r}{s+r^2-r};z \right)\]
	then we have
	\begin{align*} & \psi(p(z),zp'(z),z^2p''(z);z) \\ &=\phi\left(\frac{z(B_\kappa^c f(z))'}{B_\kappa^c f(z)},1+\frac{z(B_\kappa^cf(z))''}{(B_\kappa^c f(z))'},1+\frac{z(B_\kappa^cf(z))'''}{(B_\kappa^c f(z))''};z\right).\end{align*}
	Therefore \eqref{e16} can be rewritten as $\psi(p(z),zp'(z),z^2p''(z);z)\in\Omega$ for all $z\in \mathbb{D}$.
	Also observe that \begin{equation*}
	1+\frac{t}{s}=\frac{1-2u^2+3uv-2v+(1-v)w}{u-v}.
	\end{equation*}
	Hence the admissibility condition for $\phi\in\Phi_{C}(\Omega,q)$ is equivalent to the admissibility condition for $\psi$ as given in Definition \ref{def2}. Evidently Theorem \ref{thm1} gives  that $p(z)\prec q(z)$ and this completes the proof.		
\end{proof}
If $\Omega\ne\mathbb{C}$ is a simply connected domain, then \eqref{e16} can be written in terms of subordination so that Theorem \ref{thm5} can be reformulated and extended to the case when the behaviour of the function $q$ on the boundary of $\mathbb{D}$ is not known and the best dominants of the resulting differential subordination can be determined. The statements and proof of these results are similar to \cite[Theorem 2.3d, p.~ 30]{MR1760285} and \cite[Theorem 2.3e, p.~ 31]{MR1760285}, therefore the details are omitted.

The class of admissible functions $\Phi_{C}(\Omega,q)$ with $q(z)=e^z$ is denoted by $\Phi_{C}(\Omega,e^z)$ which is characterized by the following theorem in two-dimension.
\begin{theorem}\label{th1}
	Let $\Omega$ be any set in $\mathbb{C}$ and suppose that the function
	$\phi\colon \mathbb{C}^2\times \mathbb{D} \to\mathbb{C}$  satisfies the admissibility condition \[\phi(e^{e^{i\theta}},e^{e^{i\theta}}+me^{i\theta};z) \notin \Omega\]
	where $z\in\mathbb{D}$, $\theta\in[0,2\pi)$ and $m \geq 1$ is a positive integer. If $f\in\mathcal{A}$ satisfies the relation
	\begin{equation*}
	\left\{\phi\left(\frac{z(B_{\kappa}^cf(z))'}{B_{\kappa}^cf(z)}, 1+\frac{z(B_{\kappa}^cf(z))''}{(B_{\kappa}^cf(z))'};z\right)\colon z\in\mathbb{D}\right\}\subset \Omega
	\end{equation*}
	then $B_\kappa^c f\in\se$.
\end{theorem}
As an application, we provide the following two examples that illustrate Theorem \ref{th1}. This, in turn, give sufficient conditions in the form of differential inequalities for a function $B_{\kappa}^cf$ to be in $\se$.
\begin{example}\label{ex3.6}
	Define a function $\phi(u,v;z):=(1-\alpha)u+\alpha v$ where $\alpha> 1/e$ and let $h:\mathbb{D}\to \mathbb{C}$ be defined as
	\[h(z)=1+\frac{(\alpha e-1)}{e}z.\] Then $\Omega=h(\mathbb{D})=\{w\in\mathbb{C} \colon|w-1|<(\alpha e-1)/e \}$. We want to show that $\phi\in\Phi_{C}(\Omega,e^z)$. For $z\in\mathbb{D}$, $\theta\in[0,2\pi)$ and $m \geq 1$, consider
	\begin{align*}
	|\phi(e^{e^{i\theta}},e^{e^{i\theta}}+me^{i\theta};z)-1|^2
	&=\left|e^{e^{i\theta}}+\alpha m {e^{i\theta}}-1\right|^2\\
	& =(e^{\cos\theta}\cos(\sin\theta)+\alpha m \cos\theta -1)^2\\
	&\quad+(e^{\cos\theta}\sin(\sin\theta)+\alpha m \sin\theta )^2=:\ell_1(\theta).
	\end{align*}
	The function $\ell_1$ attains its minimum value at $\theta=\pi$ so that
	\[|\phi(e^{e^{i\theta}},e^{e^{i\theta}}+me^{i\theta};z)-1|  \geq \alpha m -\frac{1}{e} +1\geq \alpha-\frac{1}{e}+1>\alpha-\frac{1}{e}.\]
	Therefore the admissibility condition given in Theorem \ref{th1} is satisfied. Hence we conclude that if $f\in\mathcal{A}$, $\alpha> 1/e$ and \[\left|(1-\alpha)\frac{z(B_\kappa^cf(z))'}{B_\kappa^cf(z)}+\alpha\left(1+\frac{z(B_\kappa^cf(z))''}{(B_\kappa^cf(z))'}\right)-1\right|<\alpha-\frac{1}{e}\] then $B_\kappa^cf\in \se$. In the particular case, if $f(z)=z/(1-z)$ and $\alpha=1$, the following implication holds.
	\[ \left|\frac{z\vartheta''_\nu(z)}{\vartheta'_\nu(z)}\right| <1-\frac{1}{e} \implies  \left|\log \frac{z\vartheta'_\nu(z)}{\vartheta_\nu(z)} \right|<1 \quad (z\in \mathbb{D}). \]
	By Corollary \ref{cora}, it is not possible to prove that $\vartheta_{-5/2,1,1}=z\mathcal{J}_{-5/2}$ and $\vartheta_{-5/2,1,-1}=z\mathcal{I}_{-5/2}$ are in $\se$. However, the preceding implication can be used to prove the same. To see this, note that the function
	\[\frac{z\vartheta''_{-5/2,1,1}(z)}{\vartheta'_{-5/2,1,1}(z)}=\frac{z((4+z)\cos\sqrt{z}+4\sqrt{z}\sin \sqrt{z})}{2((6-z)\cos \sqrt{z}+\sqrt{z}(6+z)\sin \sqrt{z})}\]
	lies inside the circle $|w|=1-1/e$ (see Figure \ref{fig:3}). Thus it follows that
	\[\vartheta_{-5/2,1,1}(z)=z \cos \sqrt{z}-\frac{1}{3}z^2 \cos \sqrt{z}+z\sqrt{z}\sin\sqrt{z}\in\se.\]
	Similarly, it can be shown that
	\[\vartheta_{-5/2,1,-1}(z)=z \cosh \sqrt{z}+\frac{1}{3}z^2 \cosh \sqrt{z}-z\sqrt{z}\sinh\sqrt{z}\in\se.\]
	
	\begin{figure}[h]
		\begin{center}
			\includegraphics[width=2.4in]{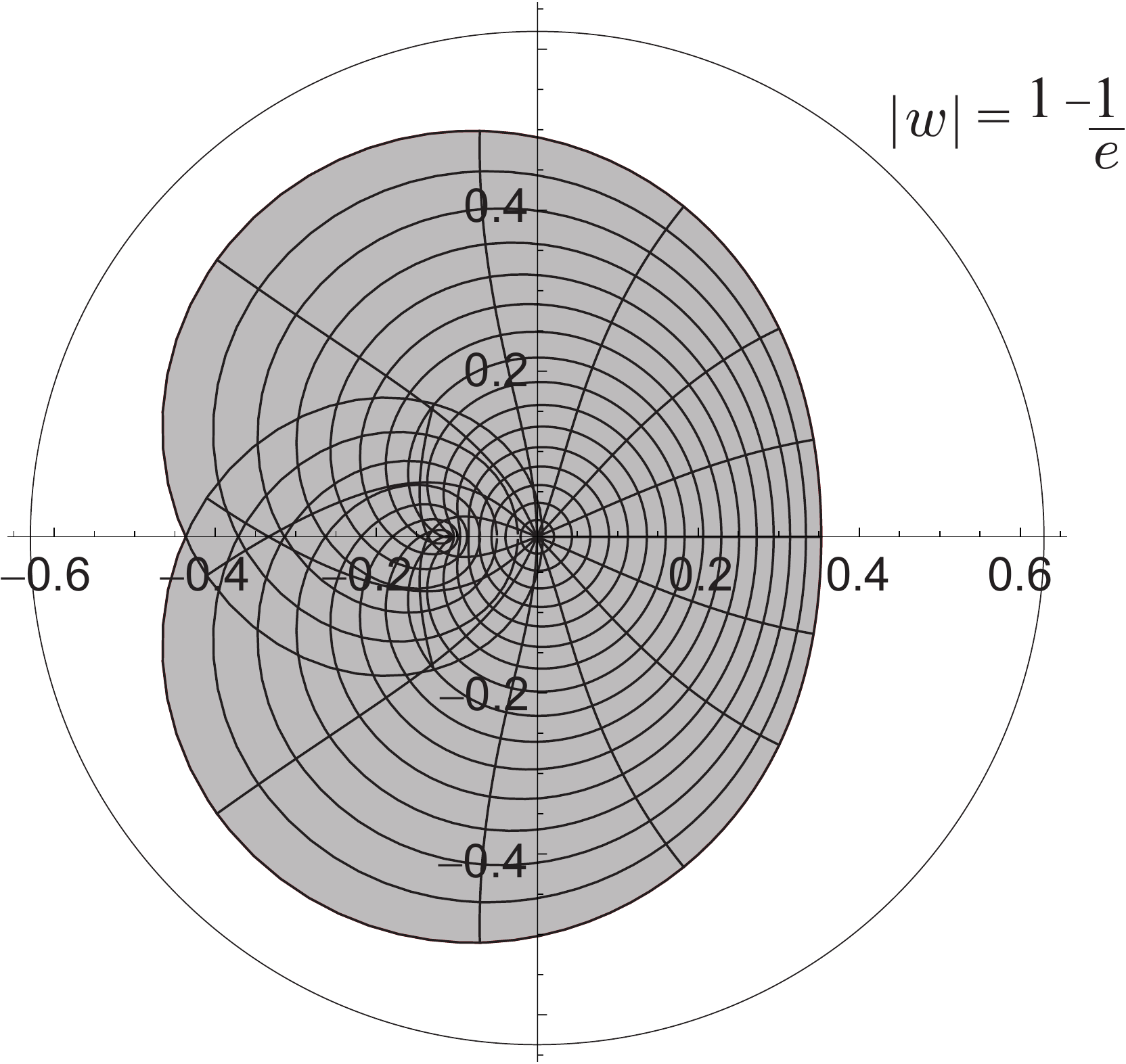}
			\caption{Image of $z\vartheta''_{-5/2,1,1}(z)/\vartheta'_{-5/2,1,1}(z)$ under $\mathbb{D}$.}
			\label{fig:3}
		\end{center}
	\end{figure}
\end{example}

\begin{example}
	Let $\phi(u,v;z):=u v$. For $z\in\mathbb{D}$, $\theta\in[0,2\pi)$ and $m \geq 1$, consider
	\begin{align*}
	|\phi(e^{e^{i\theta}},e^{e^{i\theta}}+me^{i\theta};z)-1|^2&=\left|e^{2e^{i\theta}}+ m {e^{i\theta}}e^{e^{i\theta}}-1\right|^2\\
	& =(me^{\cos\theta}\cos(\theta+\sin\theta)+e^{2\cos\theta}\cos(2\sin\theta)-1)^2\\
	&\quad\;+(me^{\cos\theta}\sin(\theta+\sin\theta)+e^{2\cos\theta}\sin(2\sin\theta))^2 \\ &=:\ell_2(\theta).
	\end{align*}
	The function $\ell_2$ attains its minimum at $\theta=\pi$ by applying the second derivative test and
	\[\min_{\theta\in[0,2\pi)}g(\theta)=g(\pi)=\left(\frac{-m}{e}+\frac{1}{e^2}-1\right)^2\]
	so that
	\[|\phi(e^{e^{i\theta}},e^{e^{i\theta}}+me^{i\theta};z)-1|\geq \frac{m}{e}-\frac{1}{e^2}+1 \geq  \frac{1}{e}-\frac{1}{e^2}+1.\]
	Therefore $\phi(e^{e^{i\theta}},e^{e^{i\theta}}+me^{i\theta};z)\not\in \Omega$ whenever $z\in\mathbb{D}$, $\theta\in[0,2\pi)$ and $m \geq 1$ where
	\[\Omega:=\left\{w\in\mathbb{C}\colon|w-1|<\frac{1}{e}-\frac{1}{e^2}+1\right\}.\]
	Hence by Theorem \ref{th1}, it follows that if $f\in\mathcal{A}$ and \[\left|\frac{z(B_\kappa^cf(z))'}{B_\kappa^cf(z)}\left(1+\frac{z(B_\kappa^cf(z))''}{(B_\kappa^cf(z))'}\right)-1\right|<\frac{1}{e}-\frac{1}{e^2}+1\] then $B_\kappa^cf\prec \se$. In particular, we have
	\[ \left|\frac{z\vartheta'_\nu(z)}{\vartheta_\nu(z)}\left(1+\frac{z\vartheta''_\nu(z)}{\vartheta'_\nu(z)}\right)-1\right|<\frac{1}{e}-\frac{1}{e^2}+1\implies \left|\log \frac{z\vartheta_\nu'(z)}{\vartheta_\nu(z)} \right|<1 \quad (z\in \mathbb{D}). \]
\end{example}
In Example \ref{ex3.6}, we proved that $\vartheta_{-5/2,1,1}\in \se$. However, the functions
\[\vartheta_{-3/2,1,1}(z)=z(\cos \sqrt{z}+\sqrt{z}\sin \sqrt{z})\]
and
\[\vartheta_{-1/2,1,1}(z)=z\cos\sqrt{z}\]
are not members of $\se$. Infact, the function $\vartheta_{-3/2,1,1}$ does not even lie in the class $\mathcal{S}$. Although the function $\vartheta_{-1/2,1,1}$ is univalent \cite[Corollary 2.9, p.~57]{MR2656410}, the figure of the image of $z\vartheta'_{-1/2,1,1}/\vartheta_{-1/2,1,1}$ depicted in Figure \ref{fig:4} clearly proves that the function $\vartheta_{-1/2,1,1}$ does not belong to $\se$. This shows that Corollary \ref{cor5} is valid only for certain range of values of $\nu\in \mathbb{R}$.
\begin{figure}[h]
	\begin{center}
		\includegraphics[width=2.4in]{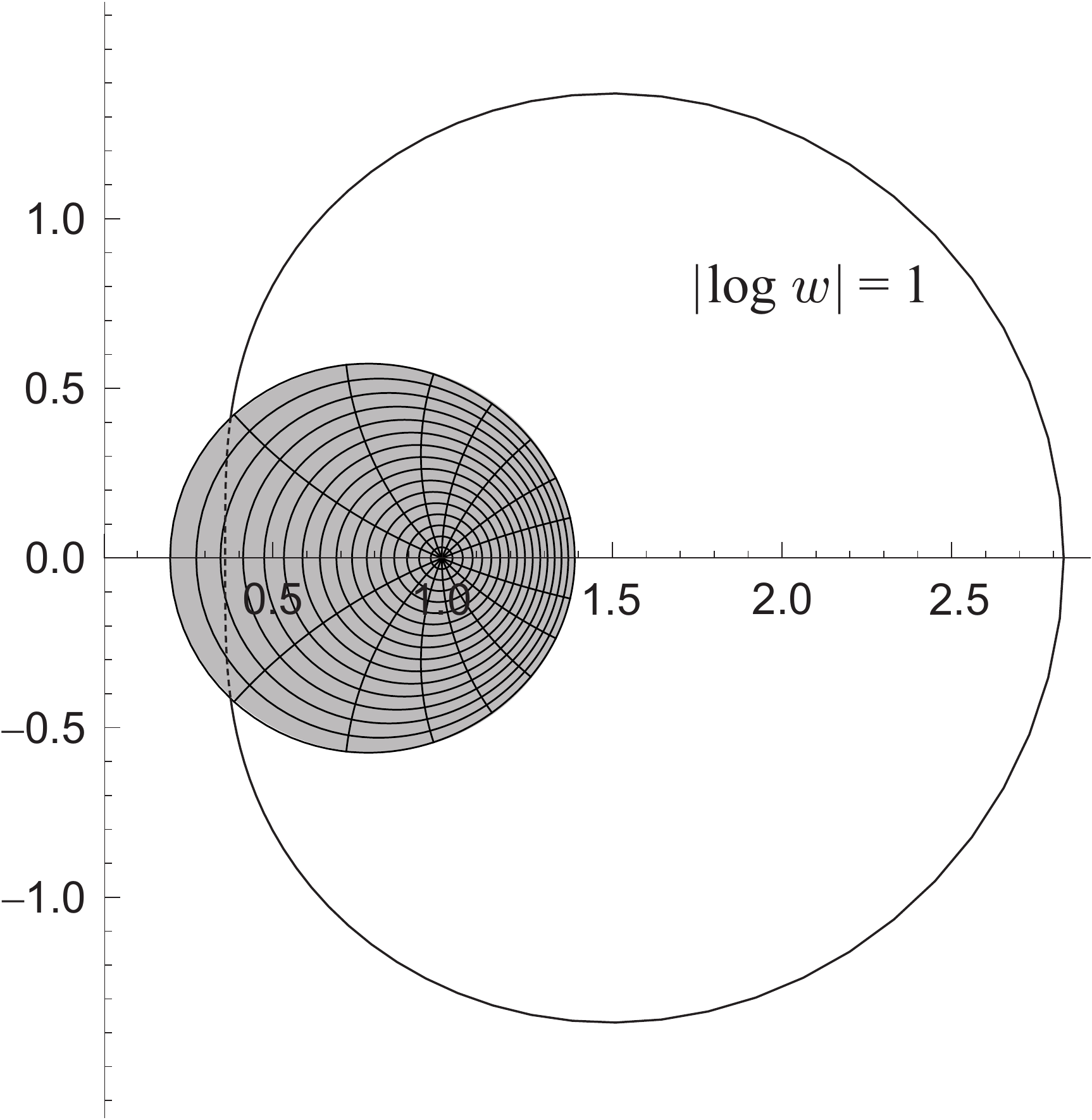}
		\caption{Image of $z\vartheta''_{-1/2,1,1}(z)/\vartheta'_{-1/2,1,1}(z)$ under $\mathbb{D}$.}
		\label{fig:4}
	\end{center}
\end{figure}

\end{document}